\documentclass{article}
\usepackage{amsmath}
\usepackage{amsfonts}
\usepackage{amssymb}
\usepackage{graphicx}
\newtheorem{theorem}{Theorem}
\newtheorem{corollary}{Corollary}
\newtheorem{lemma}{Lemma}
\newtheorem{proposition}{Proposition}
\newenvironment{proof}[1][Proof]{\noindent\textbf{#1.} }{\ \rule{0.5em}{0.5em}}
\newcommand{\A}{$\mathrm{Aut}(F_n)$}
\newcommand{\AT}{$\mathrm{Aut}(F_2)$}
\newcommand{\ON}{$\mathrm{Out}(F_n)$}
\newcommand{\OT}{$\mathrm{Out}(F_3)$}
\newcommand{\M}{$\mathrm{MCG}(S_g)$}
\begin{document}

\title{Rank gradient and cost of Artin groups and their relatives}
\author{Aditi Kar and Nikolay Nikolov \footnote{The authors acknowledge the support of EPSRC grants EP/I020276/1 (first author) and EP/H045112/1 (second author).}}
\date{}

\maketitle
\abstract{We prove that the rank gradient vanishes for mapping class groups of genus bigger than 1, \A \ for all $n$, \ON, $n \geq 3$ and any Artin group whose underlying graph is connected. These groups have fixed price 1. We compute the rank gradient and verify that it is equal to the first $L^2$-Betti number for some classes of Coxeter groups.}
\\

\\

MSC classification numbers: 20F69 (primary); 20F36 (secondary)

\section{Introduction} 
Let $G$ be a finitely generated group. We denote by $d(G)$ the minimal size of a generating set of $G$ (setting $d(\{1\})=0$) and for a subgroup $H<G$ of finite index define $$r(G,H)= (d(H)-1)/[G:H].$$
Let $H_1>H_2> \cdots $ be a chain of normal subgroups of finite index in $G$. We define the rank gradient of $G$ with respect to the chain $(H_i)$ to be 
\[ RG( G, (H_i))= \lim_{i \rightarrow \infty} r(G,H_i). \]
The notion of rank gradient was first defined by Lackenby in \cite{Lack} for the study of Kleinian groups and further investigated in \cite{AN} and \cite{AJN}.
It is not known whether the limit $RG(G,(H_i))$ depends on the choice of the chain $(H_i)$ under the condition that $\cap_i H_i=\{1\}$.
However, as shown in \cite{AJN}, when $G$ contains a normal infinite amenable subgroup or is a non-uniform lattice in a higher rank Lie group then $RG(G,(H_i))=0$ for any normal chain $(H_i)$ in $G$ with trivial intersection.

We remark that rank gradient with respect to an infinite chain in a group $G$ is a natural upper bound for its first $L^2$-Betti number $\beta^{(2)}_1(G)$. Indeed it is well known that $d(H) \geq \beta^{(2)}_1(H)$ giving that $r(G,H_i) \geq \beta^{(2)}_1(G) - [G:H_i]^{-1}$ and letting $[G:H_i] \rightarrow \infty$ we get
\begin{equation} \label{ineq} RG(G,(H_i)) \geq \beta^{(2)}_1(G).\end{equation} In particular if the rank gradient of $G$ is zero with respect to some infinite chain then the first $L^2$-Betti number of $G$ is also zero. It is not known whether the inequality (\ref{ineq}) can be strict for a normal chain $(H_i)$ with trivial intersection. 

In this note we compute the rank gradient (with respect to all chains) for several well-known families of groups: Artin groups, mapping class groups, Aut$(F_n)$, \ON\  and some Coxeter groups. A result of Gaboriau \cite{Ga2} states that the first $L^2$-Betti number vanishes for any group that contains an infinite normal finitely generated subgroup of infinite index and in this situation one can sometimes find a chain with respect to which rank gradient is zero, see \cite{KL}. This applies to Artin groups of finite type, mapping class groups and (outer-) automorphism groups of free groups. Additionally, the $L^2$-Betti numbers of Artin groups and many Coxeter groups have been computed. This paper is an attempt to compare these results with the rank gradient. In all of the groups considered here, the rank gradient is found to be independent of the normal chain and moreover is equal to $\beta^{(2)}_1(G)$. 

Our arguments yield the stronger result that these groups have \emph{fixed price}.
For definition and properties of cost and fixed price we refer the reader to \cite{Ga} and the book \cite{KM}. A key connection between rank gradient and cost is the result of Abert and Nikolov \cite{AN} that $RG(G, (H_i))= c(G, \hat G) -1$ where $c(G, \hat G)$ is the cost of the action of $G$ by left multiplication on its completion $\hat G$ with respect to the normal chain $(H_i)$ with $\cap_i H_i= \{1\}$. In this way results about cost of group actions have immediate applications to rank gradient. We remark that it is an open problem whether every countable group has fixed price.

\begin{theorem}\label{main}
Each of the groups in the list below has fixed price 1 and rank gradient zero.
\begin{itemize} \item Artin groups whose defining graphs are connected,
\item \A, $n\geq 2$,
\item \ON, $n \geq 3$, and
\item \M, $g\geq 2$. 
\end{itemize}
\end{theorem}

Theorem \ref{main} above summarises the content of Theorems \ref{artin} and \ref{free} appearing in the first half of the paper. In section \ref{cox}, we study a specific class $\mathcal{C}$ of Coxeter groups and show that the rank gradient for a group in $\mathcal{C}$ is equal to its first $L^2$-Betti number. 

Singer's conjecture predicts that the $L^2$-Betti numbers of a closed aspherical $n$-manifold are zero in all dimensions except for $n/2$. The conjecture has been verified for different classes of groups, including some right-angled Coxeter groups \cite{davies}. We suspect that more generally if $G$ is a virtual Poincare duality group of dimension at least three then the rank gradient of $G$ is zero with respect to any normal chain with trivial intersection.

\section{Basic results on rank gradient and cost} \label{basic} 
Here we collect the results about cost and rank gradient we use. While the independence of rank gradient on the chain can be deduced from the fact that a group has fixed price the direct argument for rank gradient is almost the same and we prefer to indicate these elementary techniques as well since the reader may not be familiar with the theory of cost.

\begin{proposition} \label{center} Let $G$ be a group containing a finitely generated normal subgroup $C$ whose rank gradient vanishes with respect to any chain. Then $RG (G,(H_i))=0$ for any chain of subgroups $H_i$ satisfying $(\cap_i H_i) \cap C =1$. 
\end{proposition} 

\proof Note that $C$ is infinite and $d(H) \leq d(C \cap H_i)+ d(H_i/(C \cap H_i))$. Clearly, $d(H_i/(C \cap H_i))=d(H_iC/C) \leq d(G)[G: H_iC]$. We have $$r(G,H_i) \leq  \frac{d(C \cap H_i)+d(G)[G: H_iC]}{[G:H_iC][C:C\cap H_i]}.$$ As $\cap_i (H_i \cap C) =1$ and rank gradient for $C$ vanishes with respect to any chain, we deduce the right hand side of the above inequality tends to zero, which implies that $\lim_{i \rightarrow \infty} r(G,H_i)=0$ as claimed. $\square$
\medskip

\begin{proposition} \label{amalgam} Let $G=A*_C B$ be an amalgam of two finitely generated groups over a finite subgroup $C$. Let $(H_i)$ be a normal chain in $G$ such that $C \cap (\cap_i H_i)=\{1\}$. Then \[ RG(G,(H_i))= RG( A, (A\cap H_i))+ RG(B, (B \cap H_i)) + 1/|C|.\]
\end{proposition}

\proof This is an easy computation with the Bass-Serre tree $T$ of $G$. The key observation is that for every normal subgroup $H_j$ such that $H_j \cap C=\{1\}$ we have that $H_j$ is a free product of several copies of $A \cap H_j$, $B \cap H_j$ and a free group equal to the fundamental  group of the graph $H_j \backslash T$. Applying the Grushko-Neumann theorem to $H_j$ shows that 
\[d(H_j)-1= [G:H_jA](d(A \cap H_j)-1) + [G:H_jB](d(B \cap H_j)-1) + \frac{[G:H_j]}{|C|} \] giving the result. 
$\square$
\medskip

\begin{proposition} \cite[Prop. 9]{AJN} \label{generation} Let $G$ be a group generated by two subgroups $A$ and $B$ such that $C:=A \cap B$ is infinite. Then 
\[ RG(G,(H_i)) \leq RG(A, (A\cap H_i))+ RG(B, (B\cap H_i)) \] for any normal chain $(H_i)$ in $G$ such that $[C: (C \cap H_j)] \rightarrow \infty$. 
\end{proposition} \bigskip

The cost analogues of the results above are collected in the following Proposition. It summarises some of the results of \cite{Ga} (see also \cite{KM} 31.2, 32.1, 35.1, 35.3 and 36.1).
We denote the cost of a group $G$ by $c(G)$.

\begin{proposition} \label{cost} Let $G$ be a countable group. The following results hold.\medskip

(i) If $G$ is finite then $G$ has fixed price equal to $1-|G|^{-1}$. \medskip

(ii) If $G$ is infinite amenable or has infinite centre then $G$ has fixed price 1. \medskip

(iii) Suppose that $G=\langle A,B\rangle$ is generated by two subgroups $A$ and $B$ with $A \cap B$ infinite. If both $A$ and $B$ have fixed price 1 then so does $G$.
\medskip

(iv) Suppose that $G=A \star _C B$ with amenable amalgamated subgroup $C$. If both $A$ and $B$ have fixed price then $G$ has fixed price equal to $c(A)+c(B)-c(C)$. 
\end{proposition}

\section{Artin groups}\label{Artingrp} The first family we consider are the Artin groups. We recall their definition below. Given a graph $\Gamma$ with edges labelled by integers $\geq 2$, the Artin group $A_\Gamma$ is the group with presentation given by a set of generators $a_v$ where $v$ ranges over the set of vertices of $\Gamma$ and relations 
\[ \underbrace{a_va_wa_v \cdots}_n = \underbrace{a_wa_va_w \cdots }_n \] for every edge labelled $n \in \mathbb N $ joining pair of vertices $v,w$ of $\Gamma$. 

A basic fact about Artin groups says that if one takes any subset $W$ of the vertices $V$ of the defining graph, then the subgroup generated by $W$ in $A_\Gamma$ is precisely the Artin group $A_\Omega$, where $\Omega$ is the subgraph generated by $W$ in $\Gamma$ (see \cite{Lek}). 

We note that it is unknown whether all Artin groups are residually finite. Let $G_0$ denote the intersection of all finite index subgroups of a group $G$. For an Artin group $A=A_\Gamma$ we consider normal chains $(H_i)$ in $A$ with $\cap H_i= A_0$.

Our result is:

\begin{theorem} \label{artin} Let $A=A_\Gamma$ be an Artin group and let $b$ be the number of connected components of $\Gamma$. Suppose that $(H_i)$ is a normal chain with intersection $A_0$. Then $RG(A,(H_i))=b-1$. Moreover $A$ has fixed price equal to $b$.
\end{theorem}
\proof Note that a free product of residually finite groups is residually finite, therefore in a
 free product of two groups $G=C * D$ one has $C \cap G_0=C_0$ and $D \cap G_0=D_0$. In view of Propositions \ref{amalgam} and \ref{cost} (iv) we only need to prove that $RG(A_\Gamma,(H_i))=0$ and that $A_\Gamma$ has fixed price 1 in the special case when $\Gamma$ is connected.
 
 We begin with an observation: Let $\pi : A_\Gamma \rightarrow \mathbb Z$ be the homomorphism sending all generators $a_v$ of $A$ to the canonical generator of $\mathbb Z$. Clearly $\ker \pi \geq A_0$ and therefore if $T$ is any cyclic subgroup of $A$ with $T \cap \ker \pi =\{1\}$ then $T \cap A_0 = T \cap (\cap_i H_i)=1$. In particular this applies to all vertex subgroups $\langle a_v \rangle $ of $A$. Let $e=(v,w)$ be an edge joining vertices $v$ and $w$ labelled by some integer $n>1$ of $\Gamma$. The subgroup $A_e=\langle a_v,a_w \rangle$ has infinite cyclic center $Z_e=Z(A_e)$ generated by $(a_va_w)^{n/2}$ if $n$ is even and by  $(a_va_w)^n$ if $n$ is odd. In both cases we have that $Z_e \cap \ker \pi = \{1\}$ and therefore $Z_e \cap (\cap H_i)=\{1\}$. The following Lemma is the crux of the proof of the Theorem. 
 
\begin{lemma}\label{Artin} Let $A_\Gamma$ be an Artin group, where $\Gamma$ is a connected graph. Then $A_\Gamma$ has fixed price 1. Moreover, let $(H_i)$ be a normal chain in $A_\Gamma$ such that $A_v \cap (\cap H_i)=\{1\}$ for every vertex $v\in V\Gamma$ and $Z_e \cap (\cap H_i)=\{1\}$ for every edge $e \in E\Gamma$. Then, $RG(A_\Gamma, (H_i)) = 0$. 
\end{lemma} 

\proof [Proof of Lemma \ref{Artin}] The Lemma is proved by induction on the number of vertices $n=|V\Gamma|$. If $n=1$, $A_\Gamma \cong \mathbb{Z}$ and if $n=2$, then the centre of $A_\Gamma$ is infinite cyclic. In both cases the result follows directly from Propositions \ref{center} and \ref{cost} (ii). We may assume that the number $n$ of vertices is at least 3 and that the Lemma holds for all graphs with fewer vertices. 

Let $v \in V\Gamma$. If removing $v$ disconnects the graph $\Gamma$, set $\Gamma_j, j=1,\ldots,k$ to be the connected components of $\Gamma \backslash \{v\}$. Then, the original Artin group $A_\Gamma$ is an amalgam of the subgroups $A_{\Gamma_j \cup \{v\}}$, $j=1,\ldots,k$, along the subgroup $A_v$. As $k \geq 2$, the induction hypothesis applies to each of the components $A_{\Gamma_j \cup \{v\}}$ and the Lemma follows in this case from Propositions \ref{amalgam} and \ref{cost} (iv). 

Suppose then that removing $v$ does not disconnect the graph. Let $e=(v,w)$ be an edge. Clearly, $A_e$ and $A_{\Gamma \backslash \{v\}}$ generate the group $A_\Gamma$. Moreover $A_e$ and $A_{\Gamma \backslash \{v\}}$ intersect in the subgroup $A_w$. The proof of the Lemma is now completed by applying the induction hypothesis to the subgroup $A_{\Gamma \backslash \{v\}}$ and invoking Propositions \ref{generation} and \ref{cost} (iii). 
$\square$

The Theorem provides an elementary proof for the following result which is well-known in the case of right-angled Artin groups (where all the $L^2$-Betti numbers have been determined, cf. Corollary 2 of \cite{davisleary}) but we have not been able to find a reference to the general case in the literature. 

\begin{corollary} The first $L^2$-Betti number for an Artin group $A_\Gamma$ is 0 if the underlying graph is connected. More generally, if $\Gamma$ has $b$ connected components, then the first $L^2$-Betti number of $A_\Gamma$ is precisely $b-1$. 
\end{corollary}

\section{\A, \ON \ and \M}\label{freegrp}

\begin{theorem}\label{free}
Let $G$ be one of the groups \A, $n\geq 2$, \ON, $n \geq 3$ and \M, $g\geq 2$. Then $G$ has fixed price 1 and in particular the rank gradient for $G$ is zero for any normal chain with trivial intersection.
\end{theorem}

Here, \M \ denotes the mapping class group of the closed surface $S_g$ of genus $g$. Note that $S_1$ is the torus and its mapping class group is $SL(2,\mathbb{Z})$. It is well-known that $SL(2,\mathbb{Z})$ is isomorphic to the amalgam $\mathbb{Z}/6\mathbb{Z}*_{\mathbb{Z}/2\mathbb{Z}} \mathbb{Z}/4\mathbb{Z}$ and therefore from Proposition \ref{amalgam} its rank gradient is equal to $RG(\mathbb Z / 4\mathbb Z)+ RG(\mathbb Z/ 6 \mathbb Z)+|\mathbb Z/2\mathbb Z|^{-1}=  \frac{1}{2}- \frac{1}{4}- \frac{1}{6}=\frac{1}{12}$.
\bigskip

\begin{proof}[Proof of Theorem \ref{free}] Consider first the mapping class groups. We will use the following result: \M \ is generated by some collection of Dehn twists $h_1,\ldots, h_{m}$, possibly listed with repetitions such that all the subgroups $H_i:=\langle h_i, h_{i+1} \rangle$ for $i=1,\ldots, m-1$ are isomorphic to the braid group $B_3$ on 3 strands. Indeed we may simply take the standard Lickorish generators of \M \ \cite{Lic} in the following sequence:   \[ m_1,a_1,c_1,a_2,m_2,a_2,c_2,a_3,m_3, \ldots, c_{g-1},a_g,m_g. \]
By \cite{BH} each pair of consecutive Lickorish twists generates a copy of $B_3$ since they are defined by two simple closed non-separating curves with intersection number 1.

 Since braid groups are Artin groups, by Theorem \ref{artin} $B_3 \simeq H_i$ has fixed price 1. Every two consecutive subgroups $H_i$ and $H_{i+1}$ intersect in an infinite cyclic subgroup. Therefore, by Proposition \ref{cost} (iii), the Theorem follows for all mapping class groups with $g\geq 2$. 

To deal with \A \  first we note the well-known fact that $B_n$ embeds in \A \ such that the centre $Z_n$  of $B_n$ is $B_n \cap \mathrm{Inn}(F_n)$. This shows that $B_n/Z_n$ is residually finite because \ON\ is residually finite \cite{Grossman}. For $n \geq 4$ let $A=\langle \sigma_1, \ldots, \sigma_{n-2} \rangle$ and $B= \langle \sigma_2, \sigma_3, \ldots, \sigma_{n-1} \rangle$ denote the two canonical copies of $B_{n-1}$ in $B_n$. From the description of $Z_n$ as the subgroup generated by the Garside element $\Delta_n$ we see that $A \cap Z_n= B \cap Z_n= \{1\}$ and therefore the images of $A$ and $B$ in $B_n/Z_n$ are two groups with fixed price 1 and infinite intersection. With an application of Proposition \ref{cost}(iii) we have therefore proved the following.

\begin{proposition} Let $n \geq 4$. The group $B_n/Z_n$ has fixed price 1 and rank gradient equal to 0.
\end{proposition}

Now \AT \ contains $B_4/Z_4$ as a subgroup of index two \cite{DF}. By choosing an element of infinite order in \AT \ outside $B_4/Z_4$ (i.e. any element which maps onto a matrix of infinite order and determinant -1 in $GL(2, \mathbb Z) \simeq Out(F_2)$ ) we deduce that \AT \ is generated by two groups of fixed price 1 with infinite intersection and therefore by Proposition \ref{cost} (iii) \AT \ also has fixed price 1. Let now $n>2$ and let $x_1, \ldots, x_n$ be the free generators of $F_n$. For $i=1, \ldots, n-1$, let $A_i$ be the copy of \AT \ in \A\ which preserves $\langle x_i, x_{i+1}\rangle$ and fixes all the generators except $ x_i$ and $x_{i+1}$. Consider the automorphisms $f_i \in A_i$ given by $f_i(x_i)=x_ix_{i+1}x_i^{-1}, f(x_{i+1})=x_i$. The group $D$ generated by $f_1, \ldots, f_{n-1}$ is isomorphic to $B_{n}$. 
Note that \A\ is generated by the $A_i$, and moreover $A_i \cap D \geq \langle f_i \rangle$ is infinite. Using that $A_i$ and $D$ both have fixed price 1, Theorem \ref{free} for $\mathrm{Aut}(F_n)$ follows again from Proposition \ref{cost} (iii). 

We now turn to the outer automorphism groups \ON. Again let $x_1, \ldots,x_n$ denote a set of free generators for $F_n$. If $n\geq 4$, then \ON\ is generated by the two copies of $\mathrm{Aut}(F_{n-1})$ in \ON, one which fixes $x_1$ and the other fixes $x_n$. Clearly \AT\  is contained in the intersection of these two subgroups. We can therefore invoke Proposition \ref{cost}(iii) and the fixed price 1 of \A\  to see that the Theorem holds for \ON, $n \geq 4$. 

We now deal with \OT. First we shall define three copies of \AT \ in $\mathrm{Aut}(F_3)$.

Let $X$ be the copy of \AT \ acting on $\langle x_1,x_2\rangle$ and fixing $x_3$, $Y$ be the copy of \AT \ acting on $\langle x_2,x_3\rangle$ and fixing $x_1$ and let $Z$ be the copy of \AT \ acting on $\langle x_1,x_2\rangle $ and fixing $x_1x_3$. Define $\alpha, \beta, \gamma \in \mathrm{Aut}(F_3)$ as follows:

\[ \alpha (x_1)=x_1, \ \alpha(x_2)=x_1x_2, \ \alpha(x_3)=x_3, \] \[ \beta (x_1)=x_1x_2 , \ \beta(x_2)=x_2, \ \beta(x_1x_3)=x_1x_3,\]
\[ \gamma(x_1)=x_1x_2^{-1}, \ \gamma(x_2)=x_2, \ \gamma(x_3)=x_3. \]
Now $\alpha \in X \cap Z, \beta \in Z, \gamma \in X$ and $\beta(x_3)=x_2^{-1}x_3$. The composition $\gamma \circ \beta$ fixes $x_1$ and $x_2$ and sends $x_3$ to $x_2^{-1}x_3$. 

Let us write $\bar a$,  $\bar A$ for the image of the element $a$, and respectively subgroup $A$ of $\mathrm{Aut}(F_3)$ in \OT. The groups $\bar X,\bar Y,\bar Z$ are all isomorphic to \AT \ which has fixed price 1. The intersection $\bar X \cap \bar Z \geq \langle \bar \alpha \rangle$ is infinite and therefore the group $\langle \bar X,\bar Z \rangle$ has fixed price 1.
Now $\langle \bar X,\bar Z \rangle \cap \bar Y$ contains $\bar \gamma \circ \bar \beta$ and is therefore infinite giving that $\langle \bar X,\bar Y,\bar Z \rangle$ has fixed price 1. Finally we note that $\bar X, \bar Y$ generate \OT. This completes the proof of Theorem \ref{free}. 
\end{proof}

\section{Coxeter groups}\label{cox}
As before, let $\Gamma$ be a finite graph labelled with integers bigger than 1. The Coxeter group $C_\Gamma$ is defined to be the image of the Artin group $A_\Gamma$ with the extra relations that the generators $a_v$ have order 2. 
\[ C_\Gamma= \langle A_\Gamma \  | \ a_v^2=1, \forall v \in V\rangle. \]

In many ways Coxeter groups are better understood than Artin groups, for example they are all linear groups (and so residually finite), with a concrete finite dimensional $K(\pi,1)$ complex. However their $L^2$-Betti numbers have been computed only in some special situations even in the case of right angled Coxeter groups, see for example \cite{davies}.

Our result on rank gradient is even more special. Let $\mathcal C_0$ be the class of Coxeter groups which are virtually abelian, virtually free or virtually surface groups. Let $\mathcal C$ be the smallest class of groups which contains $\mathcal C_0$ and is closed under amalgamation along subgroups $K$ with $\beta_1^{(2)}(K) = 0$ (for example $K$ can be any amenable group or a direct product of two infinite groups).
\begin{theorem} \label{coxeter} Let $G$ be an infinite group in $\mathcal C$ and let $(H_i)$ be an infinite normal chain in $G$ with trivial intersection. Then $RG(G,(H_i))=\beta_1^{(2)}(G)$.
\end{theorem}

\proof If $G \in \mathcal C_0$ the claim is straightforward: If $G$ is virtually a surface or virtually a free group then some member of the chain $N_j$ is either a surface group or a free group, in which case the equality between cost-1, rank gradient and first $L^2$-Betti number is well known (see \cite{Ga} Proposition VI.9 and \cite{Eck} Example 3.8.2). In general the definition of $\mathcal C$ says that $G$ can be obtained by a sequence of amalgamations starting from some groups in $\mathcal C_0$. By induction on the number of amalgamation steps it is enough to prove the following:
if $G=G_1 *_K G_2$ such that $\beta_1^{(2)}(K) = 0$ and $G_1,G_2$ satisfy the claim in the Theorem, then $RG(G, (N_i)) = \beta_1^{(2)}(G)$. To establish the claim, we use the following equality for $L^2$-Betti numbers (with the convention that $1/|G|=0$ if $G$ is infinite):

\begin{equation} \label{be} \beta_1^{(2)}(G)= \beta_1^{(2)}(G_1)-\frac{1}{|G_1|}+\beta_1^{(2)}(G_2)-\frac{1}{|G_2|} + \frac{1}{|K|} \end{equation}
 which holds whenever $\beta_1^{(2)}(K)=0$ , see the appendix of \cite{betti}. When $K$ is finite then the claim follows from Proposition \ref{amalgam} and the assumption on $G_1,G_2$. When $K$ is infinite then Proposition \ref{generation} gives
\[ RG(G, (N_i)) \leq RG(G_1, (G_1 \cap N_i)) + RG(G_1, (G_1 \cap N_i)) = \] \[=\beta_1^{(2)}(G_1)+\beta_1^{(2)}(G_2)=\beta_1^{(2)}(G) \leq RG(G, (N_i)) \] 
and hence again equality must hold. $\square$

\bigskip

We give an application.

\begin{theorem} Suppose that $\Gamma$ is a planar graph without circuits of length less than 6. Then $C_\Gamma$ has fixed price and for any normal chain $(N_i)$ in $C_\Gamma$ with trivial intersection we have
\[ RG(C_\Gamma, (N_i))= \beta_1^{(2)}(C_\Gamma)=\frac{|V|}{2} -1 - \sum_{e \in EV} \frac{1}{2l_e} \] where $l_e$ is the label of the edge $e$ of $\Gamma$.
\end{theorem}

\proof We show first that $\Gamma$ must have a vertex, say $v$ of valency at most two. Indeed, if every vertex has valency $\geq 3$, then the number $|E|$ of edges of $\Gamma$ is at least $3|V|/2$. On the other hand the number of regions of the plane cut out by $\Gamma$ is at most $|E|/3$ (because every region has at least 6 edges on the boundary). Now Euler's formula 
$1=|V|-|E|+|F| \leq |V| -2|E|/3 \leq 0$ derives a contradiction.

Now suppose that $v$ has valency 2 and let $w_1,w_2$ be the two neighbours of $v$. Set $A=\langle a_v,a_{w_1},a_{w_2} \rangle$, a Coxeter group whose graph is the two edges $e_1=(v,w_1)$ and $e_2=(v,w_2)$. One can check, for instance, by using (\ref{be}) that the rank gradient of the virtually free group $A$ is equal to $\beta_1^{(2)}(A)=\frac{1}{2}- \frac{1}{2l_{e_1}} - \frac{1}{2l_{e_2}}$ and of course $A$ has fixed price. 

Let $B$ be the subgroup generated by all $a_u$ for $u \in V \backslash \{v\}$. Then $B$ is a Coxeter group with graph $ \Gamma'=\Gamma - \{v\}$ and $C_\Gamma$ is the amalgam of $A$ and $B$ along the intersection $\langle a_{w_1},a_{w_2} \rangle \simeq D_{\infty}$. By induction on $|V|$ we may assume that the Theorem holds for $B$, in particular
$$ \beta_1^{(2)}(B)= \frac{|V|-1}{2} - 1 -\sum_{e \in E \Gamma'} \frac{1}{2l_e}.$$

Proposition \ref{cost} (iv) shows that $C_\Gamma$ has fixed price and together with (\ref{be}) gives that the rank gradient and first $L^2$-Betti number of $C_\Gamma$ is  $\beta_1^{(2)}(A)+ \beta_1^{(2)}(B)$. The case when $v$ has valency zero or 1 is similar. 
$\square$
\bigskip

\noindent \emph{Acknowledgement:} We thank M. Davis and Y. Antolin-Pichel for helpful comments.
 \medskip

\noindent Address: Mathematical Institute, Oxford, OX1 3LB,
UK \newline email: kar@maths.ox.ac.uk; nikolov@maths.ox.ac.uk 
\end{document}